\documentclass[reqno]{amsproc}
\usepackage{tikz}
\usepackage{amstext,amsmath,amssymb,amsfonts}
\usepackage[latin1]{inputenc}
\usepackage{epsfig}
\usepackage{hyperref}
\usepackage{color}

\usepackage{geometry}
\usepackage{amsthm}

\definecolor{myurlcolor}{rgb}{0,0,0.7}

\newtheorem{rmq}{Remark}[section]
\newtheorem{dfn}{Definition}[section]
\newtheorem{lem}{Lemma}[section]
\newtheorem{thm}{Theorem}[section]

\newtheorem{com}{Commentary}[section]
\newtheorem{Ass}{Assumption}[section]
\newcommand{\bprof}{\begin{prof}}
\newcommand{\eprof}{\end{prof}}
\newenvironment{prof}[1][Proof]{\textbf{#1.} }{\ \rule{0.5em}{0.5em}}
\usepackage{cite}

\newcommand{\bea}{\begin{eqnarray}}
\newcommand{\eea}{\end{eqnarray}}
\newcommand{\beq}{\begin{equation}}
\newcommand{\eeq}{\end{equation}}
\newcommand{\enn}{\nonumber \end{equation}}
\newcommand{\beqs}{\begin{eqnarray*}}
\newcommand{\eeqs}{\end{eqnarray*}}

 \newcommand{\cE}{\mathcal{E}}
\newcommand{\cT}{\mathcal{T}}

\def\cN{{\mathcal N}}

\newcommand{\dive}{\mathop{\rm div}\nolimits}

\def\Om{\Omega}

\title[An a posteriori error analysis]
{An a posteriori error analysis for a coupled continuum pipe-flow/Darcy model in Karst aquifers: anisotropic and isotropic discretizations}
\author{Koffi Wilfrid Houedanou$^{(a,b)}$}
\email{a) khouedanou@yahoo.fr}
\address{Universit\'e d'Abomey-Calavi (UAC), Rep. du B\'enin.}
\email{b) houedanou@aims.ac.za }
\address{African Institute for Mathematical Sciences (AIMS) South Africa.}

\begin{document}

\maketitle
\begin{Large}
\begin{abstract}\normalsize
This paper presents an a posteriori error analysis for a coupled continuum pipe-flow/Darcy model in karst aquifers. 
We consider a unified anisotropic finite element discretization (i.e. elements with very large aspect ratio).
Our analysis covers two-dimensional domains, conforming and nonconforming discretizations as well as different elements.
Many examples of finite elements that are covered by analysis are presented.
From the finite element solution, the error estimators are constructed and based 
on the residual of model equations.
Lower and upper error bounds form the main result with minimal assumptions on the elements. The lower error bound is uniform with respect to 
the mesh anisotropy in the entire domain. The upper error bound depends on a proper alignment of the anisotropy of the mesh which is a common feature 
of anisotropic error estimation. 
In the special case of isotropic meshes, the results simplify, and upper and lower error bounds hold unconditionally.\\
{\bf Mathematics Subject Classification [MSC]:} 74S05, 74S10, 74S15,
74S20, 74S25, 74S30.\\
\textbf{Keywords : } Karst aquifers; Anisotropic meshes; Error estimator.
\end{abstract}

\tableofcontents
\section{Introduction}
A coupled continuum pipe-flow (CCPF) model has been developed for groundwater flow and solute
transport in a karst aquifer with conduits. Groundwater flow in conduits is simulated through pipe-flow model and
flow in fissured matrix rock is described by Darcy's law. Water
mass exchange between the two domains is modeled by a first-order exchange rate method.
Karst aquifers are very vulnerable sources of groundwater which are largely used as drinkable
and industrial water, especially as the aquifers are now being seriously threatened by increasing
contamination. Therefore, it is practical for us to study the complicated systems like Karst aquifers
for assessing groundwater risk and controlling groundwater pollution. One of the most popular
models is so called coupled continuum pipe-flow/Darcy (CCPF) model in which the conduits
embedded in the continuum matrix are simplified into a network of one-dimensional (1D) pipes 
\cite{BLS:2003,BLS:2000,BLST:2003,LHSHCT:2003,W:2010,YGH:2009}.

Generally, the flow in the porous matrix is modeled by a continuum approach using the steady Boussinesq equation \cite{BV:1987}, and the Darcy-Weisbach 
equation \cite{B:1993} is applied to the conduit flow in the tube. The matrix flow and conduit flow are coupled at the intersection by the exchange flux, 
which is determined linearly by the difference of hydraulic heads between the matrix system and the conduit system, see 
\cite{YGH:2009,S:1999,MS:1992,G:1989}.

Near the pipe-flow region (see Fig. \ref{F1} below), the derivative of the analytic solution in porous media on $y$-direction is with a low regularity. It means that the 
solution of Darcy model in the porous media domain varie significantly along the direction parallel to $y$-axis and is smooth along parallel 
to $x$-axis. Then it is better to use the anisotropic mesh with a small mesh size on the $y$-direction near pipe-line and a 
large mesh size elsewhere, which has the advantage of improving the computational accuracy and decreasing the amount of calculation 
comparing with refining grid in all directions. In \cite{WZJ:2014,WQXJ:2013,WZH:2012,XW:00,YGH:2009,LW:2015,HAB:2017}, and in the references therein, we can find a large list of contributions 
devoted to numerically approximate the solution of this interaction problem, including conforming and nonconforming methods.

A posteriori error estimators are computable quantities, expressed in terms of the discrete solution  and of the data that measure the actual discrete
errors without the knowledge of the exact solution. They are essential to design adaptive mesh 
refinement  algorithms  which equi-distribute the computational effort and optimize the approximation efficiency. 
Since the pioneering work of Babu\v{s}ka and Rheinboldt \cite{babuska:78a},   adaptive finite element methods based on 
a posteriori error estimates have been extensively investigated. To our best knowledge, there is no a posteriori error estimation for the 
CCPF/Darcy model valid for anisotropic and isotropic discretizations with finite element methods. Here, we develop such a posteriori error analysis for 
anisotropic finite elements satisfying minimal assumptions. These assumptions may be summarised as follows: the scheme is stable (not 
essential but recommended in numerical applications), the discrete space is large enough to contain the conforming $\mathbb{P}^1$ piecewise 
space and satisfies a Crouzeix-Raviart property (see below for the details). These three properties are satisfied by some standard 
finite elements like the Crouzeix-Raviart element, modified Crouzeix-Raviart elements \cite{creuse:03} and the 
$\mathbb{Q}^k (k\geqslant 2)$ element on some anisotropic meshes.

The paper is organized as follows. Section \ref{section2} introduces the problem and some notation. The discretization (as discrete formulation) 
and the general framework with minimal conditions on the mesh and on the elements are given in Section \ref{section3}. Section \ref{section4} 
is devoted to analytical tools. In Section \ref{Finite element} we present several examples of finite elements that are covered by our analysis.
The actual error bounds are given in Section \ref{Estimatorssection}. For the upper error bound, we additionally distinguish between 
conforming and nonconforming discretization. While all these considerations are made for anisotropic meshes, we simplify the 
results for the case of an isotropic discretization in Section \ref{isotropic} since even in that case we obtain new results. We offer our 
conclusion and the further works in Section \ref{conclusion}.

\section{Preliminaries and Notation}\label{section2} 
\subsection{Model problem}
For simplification, similarly to Cao et al. \cite{YGH:2009}, we suppose the porous matrix domain $\Omega^m=\Omega_{+}^m\cup \Omega_{-}^m$, with 
$\Omega_{-}^m=]0,L[\times ]-H_m,0]$ and $\Omega_{+}^m=]0,L[\times [0,H_m[$; and the conduit pipe 
$\Omega^c=]0,L[\times \{y=0\}$. $2H_m$ is the height of the matrix and $L$ is the horizontal length of the matrix and conduit (see Fig.
\ref{F1}). We set $\Omega=\Omega^m\cup\Omega^c$. For each function $v$ defined in $\Omega$, because its restriction to $\Omega^m$ or 
to $\Omega^c$ could play a different Mathematical roles (for instance their traces on $]0,L[$), we will set 
$v^m=v_{|\Omega^m}$ and $v^c=v_{|\Omega^c}$.
Thus, the two-dimensional (2D) steady-state CCPF/Darcy model in Karst aquifers can be written in the following form:
\begin{eqnarray}\label{model}
 \left\{
\begin{array}{ccccccccccccccc}\label{r}
 &-\dive (\mathbb{K}\nabla u^m)& &=& &-\alpha_{ex}(u^m-u^c)\delta_y+f^m& &\mbox{  in  } \Omega^m&\\
&-\frac{d}{dx}\left(D\frac{d u^c}{dx}\right)&  &=& &\alpha_{ex}\left({u^m}_{|_{y=0}}-u^c\right)+f^c& &\mbox{  in  } \Omega^c,&
\end{array}
\right.
\end{eqnarray}
with the Dirichlet boundary conditions, 
\begin{eqnarray}\label{boundary}
 \left\{
\begin{array}{ccccccccccccccc}\label{r}
 &u^m=g^m& &\mbox{  on  }& &\partial \Omega^m,&\\
&u^c=g^c& &\mbox{  on  }& &\partial \Omega^c,&
\end{array}\normalfont 
\right.
\end{eqnarray}
where $u^m$ and $u^c$ denote the unknown hydraulic heads in the porous matrix $\Omega^
m$ and conduit pipe $\Omega^c$, respectively.
Under the homogeneous isotropic media assumption, the hydraulic conductivity tensor 
$\mathbb{K}$ takes the form $\mathbb{K}=\mathcal{K}\mathbb{I}$.  Here, $\mathcal{K}$ is a constant, 
$\mathcal{K}=\frac{kg}{\mu}$, where $k$ is the constant matrix permeability, $\mu$ the kinematic viscosity of water, and $g$ the gravitational 
acceleration constant. The conductivity 
constant $D$ depends on the width of the conduit $d$, $D=\frac{d^3g}{12\mu}$. $f^m$ and $f^c$ represent the external source or sink terms.
$\delta_y$ is the Dirac delta function concentrated on the straight line $\left\{y=0\right\}.$ 
The nonnegative constant $\alpha_{ex}$ represents the coefficient of flux exchange
at the intersection between the matrix and conduit flow.
Physical experimental results in \cite{CH:2012,CLC:2014,Fei:2009,WKL:2014} show that the CCPF model is valid for
flows in Karst aquifers when a suitable fluid exchange coefficient $\alpha_{ex}$ is taken. We also suppose the homogeneous boundary condition, 
$g^m=0=g^c$, which can be easily extended to a general nonhomogeneous case. The system (\ref{model})-(\ref{boundary}) 
consists of an elliptic equation
governing the Darcy flow in the porous matrix region $\Omega^m$ and an embedded one-dimensional pipe-flow equation in conduit region 
$\Omega^c$.
\begin{figure}[htpb]
\begin{center}
\tikzstyle{grisEncadre}=[thick, dashed, fill=gray!20]
\begin{tikzpicture}[scale=0.85]
color=gray!100;
 \draw (1,1)--(7,1);
 \draw (1,1)--(1,5.5);
 \draw (1,5.5)--(7,5.5);
 \draw (7,1)--(7,5.5);
 \draw [line width=4.pt](1,3)--(7,3);
 \draw [grisEncadre](1,5.5) rectangle (7,3);
  \draw [grisEncadre](1,1) rectangle (7,3);
  \draw [line width=0.2pt][>=latex,->](-8,3)--(-5,3) node [right]{$x$};
 \draw [line width=0.2pt][>=latex,->](-8,3)--(-8,6) node [right]{$y$};
  
 \draw (4,3.7) node [above]{ O O O O O O O O O O O };
 \draw (4,4.2) node [above]{O O O O O O O O O O O  };
 \draw (4,4.5) node [above]{O O O O O O O O O O O };
 \draw (4,4.8) node [above]{O O O O O O O O O O };
 \draw (4,1.3) node [above]{O O O O O O O O O O };
 \draw (4,1.6) node [above]{O O O O O O O O O O};
 \draw (4,1.9) node [above]{O O O O O O O O O O};
\draw (4,2.3) node [above]{O O O O O O O O O O};
\draw (1,5.5) node [left] {$H_m$};
\draw (1,3) node [left] {$O$};
\draw (1,1) node [left] {$-H_m$};
\draw (7,3) node [right] {$L$};

\draw[>=stealth,->] [line width=0.3pt](6.5,3.3)--(8,3.3) node [right]{$\Omega^m$};
\draw[>=stealth,->] [line width=0.3pt](6.5,1.3)--(8,3.3) node [right]{$\Omega^m$};

 \draw  (4.5,3.55) node [below]{$\Omega^c$};
 \draw[line width=0.5pt](1,1)--(1,3) node[midway,above,sloped]{$\partial\Omega^m$};
 \draw[line width=0.5pt](1,1)--(7,1) node[midway,below,sloped]{$\partial\Omega^m$};
 \draw[line width=0.5pt](7,1)--(7,3) node[midway,below,right]{$\partial\Omega^m$};
 
 \draw[line width=0.5pt](1,3)--(1,5.5) node[midway,above,sloped]{$\partial\Omega^m$};
 \draw[line width=0.5pt](1,5.5)--(7,5.5) node[midway,above,sloped]{$\partial\Omega^m$};
 \draw[line width=0.5pt](7,3)--(7,5.5) node[midway,below,sloped]{$\partial\Omega^m$};
 \end{tikzpicture}
\end{center}
\caption{\footnotesize{\textbf{Two-dimensional figure of a Karst aquifers.}}}
\label{F1}
\end{figure}
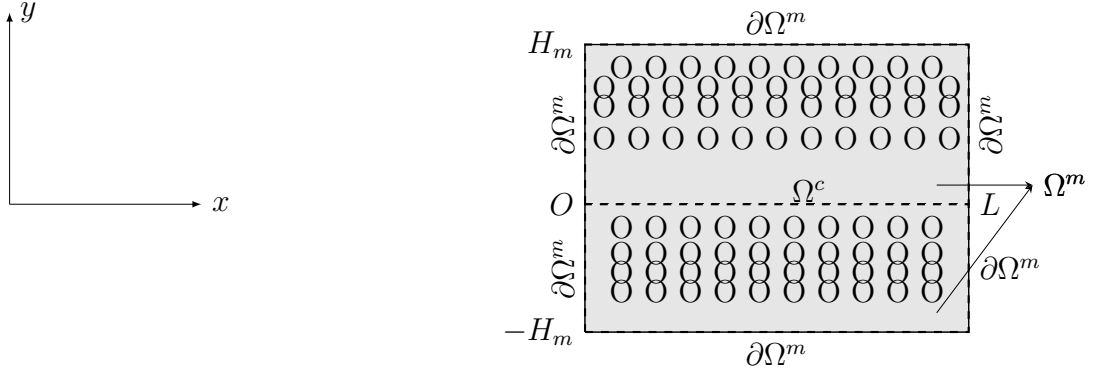

\subsection{Weak formulation}
In this section we introduce a weak formulation for the coupled 
problem given by (\ref{model}) to (\ref{boundary}). 
We begin this subsection by introducing some useful notations. If $W$ is an open bounded domain of $\mathbb{R}^2$ and $r$ 
is a non negative integer, the Sobolev space $H^r(W)=W^{r,2}(W)$ is 
defined in the usual way with the usual norm $\parallel\cdot\parallel_{r,W}$ and semi-norm $|\cdot|_{r,W}$. In particular, 
$H^0(W)=L^2(W)$ and we write $\parallel\cdot\parallel_W$ for $\parallel\cdot\parallel_{0,W}$.
Similarly we   denote by
    $(\cdot,\cdot)_{W}$  the $L^2(W)$ inner product. 
    For shortness if $W$ is equal to $\Omega:=\Omega^m\cup\Omega^c$, we will drop  the index $\Omega$, while  for any $r\geq 0$, 
$\parallel\cdot\parallel_{r,*}=\parallel\cdot\parallel_{r,\Omega_*}$, $|\cdot|_{r,*}=|\cdot|_{r,\Omega_*}$ 
and $(.,.)_*=(\cdot,\cdot)_{\Omega_*}$, for $*\in\{m,c\}$.
The space  $H_0^r(W)$ denotes the closure of $C_0^{\infty}(W)$ in $H^{r}(W)$. 

We define the Hilbert space  
$$V:=\left\{v\in H_0^1(\Omega): v=0 \mbox{  on  } \partial \Omega^c\right\},$$ with the norm 
\begin{eqnarray}
 \lVert v\lVert_{V}:=|v|_{1,m}+\lVert v\lVert_{1,c}, \mbox{   }\forall 
 v\in V.
\end{eqnarray}

Let us further introduce the bilinear form, 
$ a: V\times V\longrightarrow \mathbb{R} $  define for $u,v\in V$ by:

\begin{eqnarray}\label{bilinear}
 a(u,v)&:=& a_m(u,v)+a_c(u,v),
\end{eqnarray}
where,
\begin{eqnarray}
 a_m(u,v):=\int_{\Omega^m}\mathbb{K}\nabla u^m\cdot\nabla v^m(x,y) dx dy,
\end{eqnarray}
and 

\begin{eqnarray}\nonumber
 a_c(u,v)&:=&\int_{0}^L D\frac{du^c}{dx} \cdot\frac{dv^c}{dx} dx\\
 &+&\alpha_{ex} \int_{0}^{L}\left(u^m(x,0)-u^c(x)\right) v^m(x,0)dx\\\nonumber
 &-&\alpha_{ex}\int_{0}^{L}\left(u^m(x,0)-u^c(x)\right) v^c(x)dx.
\end{eqnarray}
In addition, we define the linear form on $V$ by 
\begin{eqnarray}\label{linear}
 F(v)&:=&(f^m,v^m)_m+(f^c,v^c)_c,
\end{eqnarray}
with $v=(v^m,v^c)$ and $f=(f^m,f^c)$.

The weak formulation of the simplified CCPF model (\ref{model})-(\ref{boundary}) can be stated as follows: find
$u\in V$ such that, 
\begin{eqnarray}\label{FF}
 a(u,v)=F(v)\mbox{     } \forall v\in V.
\end{eqnarray}

Indeed, the weak solution $u$ of simplified CCPF model (\ref{model})-(\ref{boundary}) exists and is unique. This is a straight application 
of Lax-Milgram theorem on the fact that the bilinear form $a(u,v)$ on $V\times V$ satisfies the continuity and coercivity conditions.

In Summary the following results hold:
\begin{thm}
 If $f^*\in L^2(\Omega^*)$ for $*\in\{m,c\}$, then there exists a unique solution $u\in V$ to the 
 problem (\ref{FF}).
\end{thm}
\begin{rmq} We remark that the form $a_m$ is $|\cdot|_{1,\Omega^m}$-coercive on 
 $H_0^1(\Omega^m)$ and the form $a_c$ is  $\lVert\cdot\lVert_{1,\Omega^c}$-coercive on $H_0^1(\Omega^c)$. Namely, there exist two 
 constants $c_1> 0$, $c_2> 0$ such that:
\begin{eqnarray}\label{coercivem}
 a_m(v^m,v^m)\geq c_1|v^m|_{1,m}^2, \forall  v^m\in H_0^1(\Omega^m),
\end{eqnarray}
and
\begin{eqnarray}\label{coercivec}
a_c(v^c,v^c)\geq c_2\lVert v^c\lVert_{1,c}^2, \forall  v^c\in H_0^1(\Omega^c).
\end{eqnarray}

\end{rmq}
We end this section with some notation. Let $\mathbb{P}^k$ and $\mathbb{Q}^k$
be the space of polynomials of total and partial degree not larger than $k$, respectively. In order to avoid excessive use of constants, 
the abbreviations $x\lesssim y$ and $x\sim y$ stand for $x\leqslant c y$ and $c_1 x\leqslant y\leqslant c_2 x$, respectively, with
positive constants independent of $x$, $y$ or $\cT_h$ (meshes).

\section{Anisotropic finite element method for CCPF/Darcy  model}\label{section3}
The first two sections introduce general aspects of the discretization, e.g; the finite element formulation.
Section \ref{soussection4} is then devoted to the introduction of anisotropic quantities. The general framework (mesh and general assumptions) 
will be discussed in Section \ref{soussection6}. As it turns out, the assumptions on the mesh which 
are introduced for anisotropic elements are quite weak, are standard in anisotropic a posteriori error analysis and are similar to the 
one for isotropic elements.
\subsection{Discretization of the domain $\Omega$}\label{domain}
Since the existence of Dirac delta function, the analytic solution $u^m$ of (\ref{model}) may have anisotropic behavior near
the straight line 
$\{y=0\}$. Then, we consider to use anisotropic mesh with a small mesh size on $y$-direction near the line $\{y=0\}$
and a larger mesh size  elsewhere. 

We now let $\cT_h^+$ and $\cT_h^-$ be members of families of triangulations of $\overline{\Omega_{+}^m}$ and 
 $\overline{\Omega_{-}^m}$, respectively, by triangles or rectangles $K$ of diameter $h_K$, and we 
 assume that the vertices of $\cT_h^+$ and $\cT_h^-$ coincide on the interface $[0,L].$ Also, we let 
 $h:=\max\{h_1,h_2\}$, where $h_*:=\max\{h_K: \mbox{  } K\in\cT_h^*\}$ for each $*\in\{+,-\}.$ Thereby, let $E_h$ be the corresponding 
 induced triangulations of $\overline{\Omega^c}$. Noted that according to this discretization, $E_h$ is not necessarily regular. 
 Finally $\cT_h=\cT_h^+\cup \cT_h^-$ is the triangulation on $\overline{\Omega}$. 
 
 For any $K\in \mathcal{T}_h$, we denote by $\cE (K)$ (resp. $\cN(K))$
the set of its edges (resp. vertices)  and we set 
$\cE_h=\displaystyle\bigcup_{K\in\mathcal{T}_h} \cE(K)$, $\cN_h=\displaystyle\bigcup_{K\in\mathcal{T}_h} \cN(K)$. 
For $\mathcal{A}\subset \overline{\Omega}$ we define 
$$
\cE_h(\mathcal{A}):=\left\{ E\in\cE_h: E\subset \mathcal{A}\right\} \mbox{  and  } 
\cN_h(\mathcal{A}):=\left\{\textbf{x}\in\cN_h: \textbf{x}\in \mathcal{A}\right\}.
$$

The measure of an element or edge is denoted by $|K|:=\mbox{meas}_i(K)$ and $|E|:=\mbox{meas}_{i-1}(E)$, respectively, where 
$i=2$.

For an edge $E$ of a element $K$ introduce the outer normal vector by $\textbf{n}=(n_x,n_y)^{\top}$.
Furthermore, for each segment $E$ we fix one of the two normal 
vectors and denote it by $\textbf{n}_E$. We introduce additionally the tangent vector $\textbf{t}=\textbf{n}^{\top}:=
(-n_y,n_x)^{\top}$ such that it is oriented positively (with respect to $K$). Similarly set $\textbf{t}_E:=\textbf{n}_E^{\top}$.
The superscript $\top$ denotes transposition.
For any $E\in\cE_h$
and any piecewise continuous function $\varphi$, 
we denote by $[\varphi]_E$ its jump   across $E$ in the direction of $\textbf{n}_E$:
\begin{eqnarray*}
 [\varphi]_E(x):=
 \left\{
\begin{array}{cccccc}\label{r}
&\displaystyle\lim_{t\rightarrow 0+} \varphi(x+t\textbf{n}_E)-\lim_{t\rightarrow 0+} \varphi (x-t\textbf{n}_E) &
&\mbox{for an interior edge/face $E$,}&\\
&- \displaystyle\lim_{t\rightarrow 0+}\varphi (x-t\textbf{n}_E)& &\mbox{for a boundary edge/face $E$}.&
\end{array}
\right.
\end{eqnarray*}
Note that the sign of $[\varphi]_E$ depends on the orientation of $\textbf{n}_E$. However, terms such as a gradient jump 
$[\nabla \varphi\cdot\textbf{n}_E]_E$ are independent of this orientation.

Furthermore one requires local subdomains (also known as patches). As usual, let $W_K$ be the union of all elements having a common face 
with $K$. Similarly let $W_E$ be the union of both elements having $E$ as face (with appropriate modifications for a boundary face).
By $W_{\textbf{x}}$ we denote the union of all elements having $\textbf{x}$ as node.

Later on we specify additional, mild mesh assumptions that are partially due to the anisotropic discretization.

\subsection{Discrete formulation}
We apply the finite element based on the anisotropic mesh $\cT_h$ to solve the 
CCPF model (\ref{model})-(\ref{boundary}). We assume a given approximation space $V_h$ made of polynomials on each element $K$ of the triangulation
$\cT_h$ such that $V_h^c\subset H_0^1(\Omega^c)$ (but not necessary$V_h^m\subset H_0^1(\Omega^m)$), 
where 
$V_h^*=\{v_{|\Omega^*}: v\in V_h\}$ for each $*\in\{m,c\}$. A precise description of the properties 
that this approximation space $V_h$ has to satisfy is given in Section \ref{soussection6}. 

Because the approximation space $V_h^m$ may not be included in the continuous space $ H_0^1(\Omega^m)$, we define the approximation solution by 
using the weaker bilinear form $a_h(.,.)$:
\begin{eqnarray}\nonumber
 a_h(u,v)&:=&\displaystyle\sum_{K\in\cT_h}\mathbb{K}\nabla u^m\cdot\nabla v^m (x,y) dx dy\\\nonumber
 &+&\int_{0}^L D\frac{du^c}{dx}\cdot \frac{dv^c}{dx} dx\\
 &+&\alpha_{ex} \int_{0}^{L}\left(u^m(x,0)-u^c(x)\right) v^m(x,0)dx\\\nonumber
 &-&\alpha_{ex}\int_{0}^{L}\left(u^m(x,0)-u^c(x)\right) v^c(x,0)dx.
\end{eqnarray}
Then, the finite element discretization of (\ref{FF}) is to find $u_h\in V_h$ such that 
\begin{eqnarray}\label{FFh}
 a_h(u_h,v_h)+J(u_h,v_h)=F(v_h)\mbox{     } \forall v_h\in V_h.
\end{eqnarray}
This is the natural discretization of the weak formulation (\ref{FF})  except that  the penalizing term 
$J(u_h,v_h)$ is added (only nonconforming case). These penalizing term will be specified later in the 
Section \ref{Finite element}. The space $V_h$ is equipped with the norm $\lVert\cdot\lVert_h:=\lVert\cdot\lVert_V$ if 
$V_h\subset V$ whereas the norm $\lVert\cdot\lVert_h$ on $V_h$ will be specified later in Section \ref{Finite element} for non-conforming case.

\subsection{Anisotropic quantities}\label{soussection4}
For an
element $K\in\cT_h$ we define two anisotropy vectors $\textbf{P}_{i,K}$, $i=1,2$, that reflect the main 
anisotropy directions of that element. These anisotropy vectors are defined and visualized below as well
(Figs. \ref{rectangle} and \ref{triangle} below). 
The anisotropy vectors $\textbf{P}_{i,K}$ are enumerated such that lengths are decreasing, 
i.e. 
$|\textbf{P}_{1,K}|\geqslant |\textbf{p}_{2,K}|$.
 The anisotropic lengths of an element $K$ are now defined by 
$h_{j,K}:=|\textbf{P}_{j,K}|$, $(j=1,2)$ which implies $h_{1,K}\geq h_{2,K}$.
The smallest of these lengths
is particularly important; thus 
we introduce $\displaystyle h_{\min,K}:=h_{2,K}\equiv \min_{i\in\{1,2\}}h_{i,K}$.
Finally the anisotropy vectors $\textbf{P}_{j,K}$ are arranged columnwise to define a matrix: 
\begin{eqnarray}\label{matrixstructure}
\mathbb{C}_K&:=&[\textbf{P}_{1,K},\textbf{P}_{2,K}]\in\mathbb{R}^{2,2}.
\end{eqnarray}
Note that $\mathbb{C}_K$ is orthogonal since 
anisotropy vectors $\textbf{P}_{j,K}$ are also orthogonal and 
\begin{eqnarray}
 \mathbb{C}_K^\top\cdot\mathbb{C}_K=\mbox{diag}\{h_{1,K}^2,h_{2,K}^2\}.\end{eqnarray}
 Furthermore introduce the height $h_{E,K}$ over an edge $E$ of an element $K$ by 
\begin{eqnarray}
h_{E,K} &:=&\frac{|K|}{|E|}.\end{eqnarray}
Sometimes it is more convenient to have face-related data instead of element-related data. Hence for an interior 
face $E=K_1\cap K_2$ we introduce
\begin{eqnarray*}
 h_{\min,E}:=\frac{h_{\min,K_1}+h_{\min,K_2}}{2} \mbox{       and           } h_E:=\frac{h_{E,K_1}+h_{E,K_2}}{2}.
\end{eqnarray*}
For boundary faces $E\subset \partial K$ simply set $h_{\min,E}:=h_{\min,K}$, $h_E:=h_{E,K}$. 
The last 
assumption from below (Assumption \ref{ass1}) readily implies
\begin{eqnarray}\label{Hmaillage2}
 h_E\sim h_{E,K_1}\sim h_{E,K_2} \mbox{   and   } h_{\min,E}\sim h_{\min,K_1}\sim h_{\min,K_2}.
\end{eqnarray}
\begin{figure}[http]
\centering
\begin{center}
\begin{tikzpicture}
 \draw [line width=1pt][>=stealth,->](1,1)--(1,3)node [above] {$\textbf{P}_{2,K}$};
 
 \draw [line width=0.2pt][>=latex,->](-8,1)--(-6,1)node [right] {$x$};
 \draw [line width=0.2pt][>=latex,->](-8,1)--(-8,3)node [above] {$y$};
 \draw [line width=1pt][>=stealth,->](1,1)--(5,1)node [right] {$\textbf{P}_{1,K}$};
 \draw [line width=0.5pt](1,3)--(5,3);
 \draw [line width=0.5pt](5,1)--(5,3);
 \draw [line width=0.5pt](1.4,1) arc (0:90:0.43);
 \draw (1.2,1.39) node [below] {$\bullet$};
 
 
\end{tikzpicture}
\caption{\footnotesize{\Large{Notation of rectangle $K$}}}
\label{rectangle}
\end{center}
\end{figure}
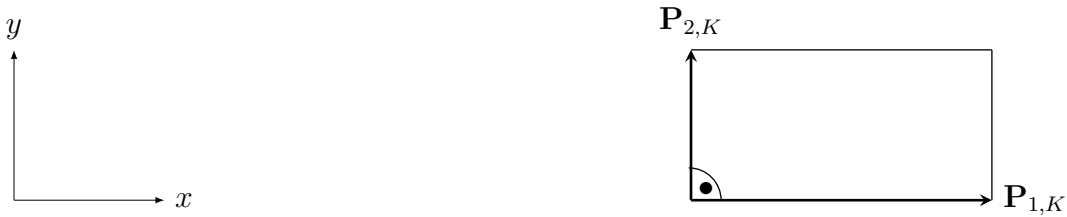
\begin{figure}[http]
\centering
\begin{center}
\begin{tikzpicture}[scale=0.7]
 \draw [line width=1.2pt] [>=stealth,->](0,0)--(8,0) node [midway,below] {$\textbf{P}_{1,K}$};
 \draw (0,0) node [left] {$P_0$};
 \draw [line width=1.2pt] [>=stealth,->](0,0)--(8,0) node [below] {$P_1$};
 \draw [line width=1.2pt] [>=stealth,->] (1.86,0)--(1.8,2.4) node [midway,right] {$\textbf{P}_{2,K}$};
 \draw [line width=1.2pt] [>=stealth,->] (1.86,0)--(1.8,2.4) node [above] {$P_2$};
 \draw [line width=0.5pt](0,0)--(1.8,2.4);
 \draw [line width=0.5pt](8,0)--(1.8,2.4);
 
 \draw [line width=0.2pt][>=latex,->](-8,0)--(-5,0)node [right] {$x$};
 \draw [line width=0.2pt][>=latex,->](-8,0)--(-8,2.4)node [above] {$y$};
 \draw (2.2,0.6) node [below] {$\bullet$};
  \draw [line width=0.5pt](2.7,0) arc (0:90:0.8);
\end{tikzpicture}
\caption{\footnotesize{\Large{Notation of triangle $K$}}}
\label{triangle}
\end{center}
\end{figure}
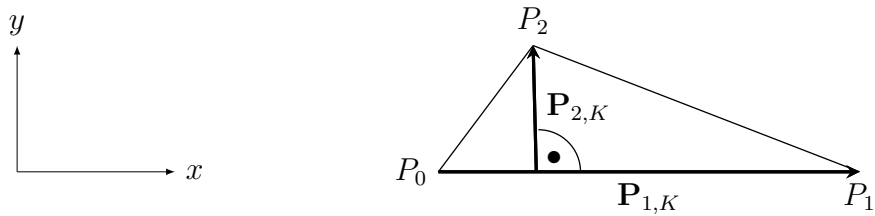
\subsection{Relation between anisotropic mesh and anisotropic function}\label{soussection5}
When investigating a residual error estimator for anisotropic meshes, we want to employ the same basic principles as for isotropic meshes. 
More precisely, a certain kind of interpolation error estimates is to be derived first. With its help, the finite element error is
then bounded globally from above. 

Proceeding this way, we naturally use different and more technical methods than for isotropic meshes. But even more important, the results of
isotropic meshes can not be transferred identically to anisotropic meshes.  A certain factor appears now both at the 
interpolation error estimates (see Section \ref{Clementsection}) and the finite element error estimate (cf. Section \ref{Estimatorssection}).
This factor is related to how good the chosen anisotropic mesh corresponds to the anisotropic function under consideration. 
Basically, the better this correspondence the smaller the factor (but always $\geqslant 1$), and the better the estimate (in a meaning that 
is to be specified later on). The importance of an anisotropic mesh that corresponds to an anisotropic function can be described and 
interpreted in different ways (Ref. \cite[Page 33]{Kunert:98}).

We present now the definition of an alignment measure which measures the alignment of mesh and function.

\begin{dfn}\normalfont(\textbf{Alignment measure} $m_1$) 
 Let $v\in H^1(\Omega)$ be an arbitrary non-constant function. 
 Define the matching function $m_1(.,.): H^1(\Omega)\times \cT_h\longrightarrow \mathbb{R}$ 
 by \cite{ESK:04,Kunert:98,creuse:03}
 \begin{eqnarray}\label{mesure1}
 m_1(v,\cT_h)&:=&
\frac{\left(\displaystyle\sum_{K\in\mathcal{T}_h}h_{\min,K}^{-2}
\parallel \mathbb{C}_K^{\top}\nabla v\parallel_K^2\right)^{1/2}} {\parallel \nabla v\parallel_{\Omega}}.
\end{eqnarray}
\end{dfn}

\begin{com}\normalfont (\textbf{Alignment measure})
 For a better understanding we discuss here the behaviour of the alignment measure. The structure 
of the matrix $\mathbb{C}_K$ from (\ref{matrixstructure}) readily gives the crude bounds,
\begin{eqnarray}
 1\leqslant m_1(v,\cT_h)\leqslant \displaystyle\max_{K\in\cT_h}\frac{h_{\max,K}}{h_{\min,K}},
\end{eqnarray}
where $h_{\mbox{max},K}\equiv h_{1,K}$ temporarily denotes the largest element dimension. Although this bound is pratically useless, it 
implies an interesting by-product for isotropic meshes. There one concludes $m_1(v,\cT_h) \sim 1$, and the alignment measure merges with 
other constants and thus "vanishes".

For anisotropic meshes, the term $\mathbb{C}_K^{\top}\nabla v$ of (\ref{mesure1}) contains directional derivatives along  
the main anisotropic directions 
$\textbf{P}_{i,K}$ of the element $K$ [since  $\mathbb{C}_K=[\textbf{P}_{1,K},\textbf{P}_{2,K}]$, see (\ref{matrixstructure})]. 
Consider first anisotropic elements that are aligned with an anisotropic function $v$. Then the long anisotropic element
direction $\textbf{P}_{1,K}$ is associated with a small directional derivative $\textbf{P}_{1,K}^{\top}\cdot\nabla v$.
Conversely, the short direction $\textbf{P}_{2,K}$ has a comparatively large directional derivative 
$\textbf{P}_{2,K}^{\top}\cdot\nabla v$. Consequently the numerator and denominator of $m_1(.,.)$ will be balanced, and 
$m_1(.,.)\sim 1$. Supplementary details are given in Ref. \cite{Kunert:02}.

If the anisotropic mesh is not aligned with an anisotropic function $v$, then similar considerations imply that the numerator and 
denominator of $m_1(.,.)$  are no longer balanced , and thus $m_{1}(.,.)>> 1$.

Summarising, the better the anisotropic mesh $\cT_h$ is aligned with an anisotropic function $v$, the smaller 
$m_1(.,.) $ will be . This results in sharper error bounds.
\end{com}

\subsection{Requirements on the mesh and the elements}\label{soussection6}
\begin{Ass}\normalfont (\textbf{Mesh assumptions in $\Omega$}) \label{ass1}
 Let $a_1, \ldots, a_n$ be the nodes of the triangulation $\cT_h$. In addition to the usual conformity conditions 
of the mesh (see \cite [Chapter 2]{Ciarlet:78}) we demand the following assumptions.
\begin{itemize}
 \item The number of element that contain the node $a_j$ is bounded uniformly.
 \item The dimensions of adjacent element must not change rapidly, i.e. 
 \begin{eqnarray}\label{Hmaillage1}
  h_{i,K'\sim h_{i,K}}\mbox{       } \forall  K,K' \mbox{  with  } K\cap K'\neq \emptyset, i=1,2.
 \end{eqnarray}
 \end{itemize}
 
\end{Ass}
\begin{Ass}\normalfont (\textbf{General assumptions})\label{ass2}
  In our analysis, a Cl\'ement type operator $\mbox{I}_{\mbox{Cl}}^0$ plays a vital role.
Although the precise definition will be postponed until Section \ref{Clementsection}, we briefly describe the image space of this operator. 
Roughly speaking, its functions are continuous and piecewise linear for an element $K\in\cT_h$.
From now on, we use the notation 
$$V_{\mbox{Cl}}^0:=\left[Im\left(I_{\mbox{Cl}}^0\right)\right],$$
 for the Cl\'ement interpolation space in $\Omega$. The general condition is now as follows. \\
 (H.1) The 
 space $V_h$ is large enough such that it contains the Cl\'ement interpolation space $V_{\mbox{Cl}}^0$, that is, 
 $V_{\mbox{Cl}}^0\subset V_h\cap H_0^1(\Omega)$. \\
 (H.2) In order to obtain robust discrete solution, the elements have to be stable (i.e. the form bilinear $a_h$ must be 
 $\lVert\cdot\lVert_h-$coercive on $V_h$).
\end{Ass}\mbox{}\\
\textbf{Crouzeix-Raviart Property for Nonconforming Approximation.} For nonconforming approximation we require the 
"Crouzeix-Raviart" property:
\begin{eqnarray}\label{crouzeix}
 (\mbox{CR}): \int_E[u_h]_E&=&0, \forall E\in\cE_h.
\end{eqnarray}

\section{Analytical tools}\label{section4}

\subsection{Cl\'ement interpolation}\label{Clementsection}
For the analysis we require some interpolation operator that maps a function from $H_0^1(\Omega)$ to some continuous, piecewise polynomial 
function $V_{\mbox{Cl}}^0$. Hence Lagrange interpolation is unsuitable, but Cl\'ement like interpolation techniques have proven to be useful. 
The image space $V_{\mbox{Cl}}^0$ will be given by means of its basis functions. 
To this end denote by $F_K$ temporarily that affine linear 
transformation that maps the reference element $\overline{K}$ into the actual element $K$. 
For simplicity we describe the interpolation for 
scalar functions.

The basis function $\phi_j$ associated with a node $\textbf{x}_j$ is now uniquely determined by the condition
\begin{eqnarray}
 \phi_j(\textbf{x}_i)&=& \delta_{i}^j \mbox{    }  \forall \textbf{x}_i\in \mathcal{N}_h(\Omega).
\end{eqnarray}

Then $V_{\mbox{Cl}}^0$ is defined as the space spanned by the functions $\phi_j$, for all interior nodes $\textbf{x}_j\in \cN_h(\Omega)$. 
Equivalently, it can be expressed as 
\begin{eqnarray}
 V_{\mbox{Cl}}^0&:=&\left\{v_h\in C^0(\Omega):  \mbox{  } {v_h}_{|K}\circ F _K\in \mathcal{P}^1
 (\bar{K}), 
 \forall K\in \cT_h\right\}\cap H_0^1(\Omega),
\end{eqnarray}
with $\mathcal{P}^1(\bar{K})=\mathbb{P}^1(\bar{K})$ if $K$ is triangle and $\mathcal{P}^1(\bar{K})=\mathbb{Q}^1(\bar{K})$ if $K$ is rectangle.
$F_K$ is defined as above.

Next, the Cl\'ement interpolation operator will be defined via the basis functions $\phi_j\in V_{\mbox{Cl}}^0$.
\begin{dfn}\normalfont
 (\cite[Section 4]{ESK:04} \textbf{Cl\'ement interpolation operator}) 
 Consider an interior node $\textbf{x}_j\in \cN_h(\Omega)$ and the patch $w_{\textbf{x}_j\equiv \mbox{supp}(\phi_j)},$ cf. 
 Section \ref{domain}. Define the local $L^2$ projection operator $P_j: L^2(w_{\textbf{x}_j})\longrightarrow 
 \mathbb{P}^{0}(w_{\textbf{x}_j})$ by
 \begin{eqnarray}
  \int_{w_{\textbf{x}_j}} (v-P_jv)w=0 \mbox{   } \forall w\in \mathbb{P}^0(w_{\textbf{x}_j}).
 \end{eqnarray}
Then define the Cl\'ement interpolation operator $I_{\mbox{Cl}}^0: H_0^1(\Omega)\longrightarrow V_{\mbox{Cl}}^0\subset H_0^1(\Omega)$ by
\begin{eqnarray}
 \mbox{I}_{\mbox{Cl}}^0v:=\displaystyle\sum_{\textbf{x}_j\in \cN_h(\Omega)}P_j(v)(\textbf{x}_j)\phi_j.
\end{eqnarray}

\end{dfn}

We can prove the following interpolation estimates \cite{creuse:03, ESK:04} (see also, \cite{H:2015,HA:2016}):
\begin{lem}\label{Interpolationanisotrope}\normalfont
 For all $v\in H_0^1(\Omega)$, we have:
\begin{eqnarray}\label{cle1}
 \sum_{K\in\mathcal{T}_h}h_{\min,K}^{-2}\parallel v-I_{\mbox{Cl}}^0 v\parallel_K^2&\lesssim &
m_1^2(v,\mathcal{T}_h)\parallel \nabla v\parallel_{\Omega}^2,\\\label{cle2}
\sum_{E\in\cE_h(\bar{\Omega})}\frac{h_E}{h_{\min,E}^2}\parallel v-I_{\mbox{Cl}}^0v\parallel_E^2
&\lesssim& m_1^2(v,\mathcal{T}_h)\parallel \nabla v\parallel_{\Omega}^2.
\end{eqnarray}
\end{lem}

\subsection{Bubble functions, extension operator, inverse inequalities}\label{bulle}
For the analysis we require bubble functions and extension operators that satisfy certain properties. We 
start with the reference element $\overline{K}$ and define an element bubble function $b_{\overline{K}}\in C(\overline{K})$. 
We also require an edge bubble function $b_{\overline{E},\bar{K}}\in C(\overline{K})$ for a face
$\overline{E}\subset \partial K$.  Without loss of generality assume that $\overline{E}$ is on the 
$\overline{x}$ axis. 
Furthermore an extension operator $F_{\mbox{ext}}: C(\overline{E})\longrightarrow C(\overline{K})$ will be necessary that acts on 
some function $v_{\overline{E}}\in C(\overline{E})$. The table below give the definitions in each case (i.e. triangle or rectangle element).
\begin{center}
\begin{tikzpicture}
 \draw (0,0) grid[step=5.] (15,10);
 \draw (0,8)--(15,8);
 \draw (2,8.5) node [above]{$\mbox{Ref. element $\bar{K}$}$};
 \draw (7,8.5) node [above]{$\mbox{Bubble functions}$};
 \draw (12,8.5) node [above]{$\mbox{Extention operator}$};
 \draw (7.5,7) node [above]{$b_{\bar{K}}:=3^3\bar{x}\bar{y}(1-\bar{x}-\bar{y})$};
 \draw (7.5,5.5) node [above]{$b_{\bar{E},\bar{K}}:=2^2\bar{x}(1-\bar{x}-\bar{y})$};
 \draw (12.5,6) node [above]{$F_{ext}(v_{\bar{E}})(\bar{x},\bar{y}):=v_{\bar{E}}(\bar{x})$};
 \draw (7.5,3) node [above]{$b_{\bar{K}}:= 2^4\bar{x}(1-\bar{x})\bar{y}(1-\bar{y})$};
 \draw (7.5,1) node [above]{$b_{\bar{E},\bar{K}}:=2^2\bar{x}(1-\bar{x})(1-\bar{y})$};
 \draw (12.5,2.5) node [above]{$F_{ext}(v_{\bar{E}})(\bar{x},\bar{y}):=v_{\bar{E}}(\bar{x})$};
 \draw [line width=1pt](2,6)--(3.5,6);
 \draw [line width=1pt](2,6)--(2,7.5);
 \draw [line width=1pt](2,7.5)--(3.5,6);
 \draw (2.5,5.4) node [above]{$0\leqslant \bar{x},\bar{y}$};
 \draw (2.6,5.0) node [above] {$\bar{x}+\bar{y}\leqslant 1$};
 \draw [line width=1pt](2,2)--(2,4);
 \draw [line width=1pt](2,2)--(4,2);
 \draw [line width=1pt](4,2)--(4,4);
 \draw [line width=1pt](2,4)--(4,4);
 \draw (1.7,1.2) node [above]{$0\leqslant \bar{x},\bar{y}\leqslant 1$};
\end{tikzpicture}\mbox{}\\\vspace*{0.5cm}
\tablename  \mbox{     }1: {\footnotesize{\textbf{ \mbox{      
Bubble functions and  extension operator on  $\bar{K}$       }}}}
\end{center}
\label{Ou}
\mbox{}\\
The element bubble function $b_{K}$ for the actual element $K$ is obtained  simply by the corresponding affine linear
transformation. Similarly the edge/face bubble function $b_{E,K}$ is defined. Later on an edge/face bubble function $b_E$
is needed on the domain $w_E=K_1\cup K_2.$ This is achieved by an elementwise definition, i.e. 
\begin{eqnarray*}
 {b_{E}}_{|K_i}:= b_{E,K_i}, i=1,2.
\end{eqnarray*}
Analogously the extension operator is defined for functions $v_E\in C(E)$. By the same elementwise definition we obtain 
$F_{\mbox{ext}}(v_E)\in C(w_E).$ With these definitions one easily checks
\begin{eqnarray*}
 b_K=0 \mbox{  on } \partial K, \mbox{  } b_E=0 \mbox{  on  } \partial w_E, \mbox{  } 
 \parallel b_K\parallel_{\infty}=\parallel b_E\parallel_{\infty}=1.
\end{eqnarray*}
Next, one requires the so-called inverse inequalities. They can only be expected to hlod in some finie-dimensional
space. The choice $\mathbb{P}^k$ covers all relevant case of our analysis.
\begin{lem}\normalfont
 (\textbf{Equivalences/Inverse
inequalities for bubble functions})\\ Let $E\in\cE(K)$ be an edge of an element $K$. 
 Consider $v_K\in \mathbb{P}^{k_0}(K)$ and $v_E\in \mathbb{P}^{k_1}(E)$ . 
 Then the following equivalences/inequalities hold. The inequality constants depend on the polynomial
 degree $k_0$ or $k_1$ but not on $K$, $E$ or $v_K$, $v_E.$
 \begin{eqnarray}\label{I1}
  \parallel v_Kb_K^{1/2}\parallel_K&\sim& \parallel v_K\parallel_K\\\label{I2}
\parallel \nabla(v_Kb_K^{1/2})\parallel_K&\lesssim& h_{\min,K}^{-1}\parallel v_K\parallel_K\\\label{I3}
\parallel v_Eb_E^{1/2}\parallel_E&\sim& \parallel v_E\parallel_E\\\label{I4}
\parallel F_{\mbox{ext}}(v_E)b_E\parallel_K&\lesssim& h_{E,K}^{1/2}\parallel v_E\parallel_E\\\label{I5}
\parallel\nabla(F_{\mbox{ext}}(v_E)b_E)\parallel_K&\lesssim& h_{E,K}^{1/2}h_{\min,K}^{-1}\parallel v_E\parallel_E.
 \end{eqnarray}
\end{lem}
\begin{proof}
 Reference \cite{GK:2000}.
\end{proof}

\section{Examples of Finite elements }\label{Finite element}
\subsection{Crouzeix-Raviart elements I}
For a triangulation of $\Omega$ consisting of triangles in $2D$, we approximate the exact solution $u$ in the Crouzeix-Raviart finite 
element space \cite{ANJ:2001, CR:1973, GR:1986}, namely,
\begin{eqnarray*}
 V_h&:=& \left\{v_h\in L^2(\Omega): \mbox{   } {v_h}_{|K}\in \mathbb{P}^1(K), \forall K\in \cT_h, \int_E [v_h]_E=0 \forall  E\in\cE_h\right\}\cap H_0^1(\Omega^c).
\end{eqnarray*}
The bilinear form $J(.,.):V\cup V_h\longrightarrow \mathbb{R}$ is defined here as follows:
\begin{eqnarray}\label{J}
 J(u_h,v_h):= \displaystyle\sum_{E\in \cE_h(\Omega^m)}\frac{h_E}{h_{\min,E}^2}\int_E[u_h^m]_E\cdot[v_h^m]_E, \mbox{   } u_h,v_h\in V\cup V_h.
\end{eqnarray}
We are now able to define the norm on $V_h$:
\begin{eqnarray}\label{norm}
 \lVert v\lVert_h:=\left(\displaystyle\sum_{K\in\cT_h}|v_h^m|_{1,K}^2+\lVert v_h^c\lVert_{1,\Omega^c}^2+J(v_h^m,v_h^m)\right)^{1/2}.
\end{eqnarray}
These
Crouzeix-Raviart elements are nonconforming (i.e. $V_h\nsubseteq V$).
It is clear that the bilinear form $a_h$ is $\lVert \cdot\lVert_h$-coercive on $V_h$
independently of the aspect ratio of the element $K$ of the triangulation, which means that (H.2) is
valid. Since in this case we have $V_{\mbox{Cl}}^0=H_0^1(\Omega)\cap V_h$, the assumption (H.1) holds. In addition, the 
Crouzeix-Raviart elements satisfy the condition (CR) by definition.
\subsection{Crouzeix-Raviart elements II} Here we restrict to a triangulation of $\Omega$ made of rectangles. Due to the condition 
(H.1) we actually need to modify the finite element given in \cite{AN:2001,ANJ:2001}. On the reference rectangle $\bar{K}=(0,1)^2$ we define 
\begin{eqnarray}
 \bar{\mathbb{Q}}^{1+}:=\mbox{span}\{1,\bar{x},\bar{y},\bar{x}\bar{y},\bar{y}^2\}.
\end{eqnarray}
As degree of freedom (i.e. functionals of $\Sigma$) we take
\begin{eqnarray*}
 \bar{\theta}_i(q):=\int_{\bar{E}_i}q, \mbox{  } i=1,\ldots,4,\mbox{    } \bar{\theta}_5(q):=\int_{\bar{K}} \bar{q}_5q,
\end{eqnarray*}
where $\bar{E}_i$ are the four edges of $\bar{K}$, and $ \bar{q}_5$ is the polynomial defined by 
\begin{eqnarray}
 \bar{q}_5(\bar{x},\bar{y}):=3(2\bar{x}-1)(2\bar{y}-1).
\end{eqnarray}
One readily checks that the triplet $\left(\bar{K},\bar{\mathbb{Q}}^{1+},\{\bar{\theta}_i\}_{i=1}^5\right)$ 
is a finite element \cite[Page 75]{AE05} 
associated basis is given by $\{\bar{q}_i\}_{i=1}^5$, where
$$ \bar{q}_1(\bar{x},\bar{y}):=1-4\bar{y}+3\bar{y}^2, \mbox{   } \hspace*{1cm}
\bar{q}_2(\bar{x},\bar{y}):=-2\bar{y}+3\bar{y}^2, $$
$$\hspace*{1cm}\bar{q}_3(\bar{x},\bar{y}):=\frac{1}{2}-\bar{x}+3\bar{y}-3\bar{y}^2, \mbox{   } \hspace*{1cm}
\bar{q}_4(\bar{x},\bar{y}):=-\frac{1}{2}+\bar{x}+3\bar{y}-3\bar{y}^2.
$$
The edges are $\bar{E}_1=(0,1)\times \{0\},\bar{E}_2=(0,1)\times \{1\} , \bar{E}_3=\{0\}\times (0,1)$ and 
$\bar{E}_4=\{1\}\times (0,1)$.

The finite element $\left(K,\mathbb{Q}^{1+},\{\theta_i\}_{i=1}^5\right)$ on the actual anisotropic rectangle $K$ is obtained by 
a standard affine transformation from $\left(\bar{K},\bar{\mathbb{Q}}^{1+},\{\bar{\theta}_i\}_{i=1}^5\right)$ such that $\bar{y}$ is 
mapped onto the stretching direction of the rectangle.

The space $V_h$ is defined by 
\begin{eqnarray*}
 V_h&:=& \left\{v_h\in L^2(\Omega): \mbox{   } {v_h}_{|K}\in\mathbb{Q}^{1+} , \forall K\in \cT_h, \int_E [v_h]_E=0 \forall  E\in\cE_h\right\}\cap H_0^1(\Omega^c).
\end{eqnarray*}
The bilinear form $J(.,.):V\cup V_h\longrightarrow \mathbb{R}$ is defined as in (\ref{J}). The discrete norm $\lVert\cdot\lVert_h$ is also 
defined as in (\ref{norm}). The first condition (H.1) clearly holds: for rectangles, 
$V_{\mbox{Cl}}^0$ consists of continuous and piecewise bilinear functions. In addition, the last assumption (H.2) and the condition (CR) 
are satisfied trivially\cite{creuse:03}.
Note that the condition (H.1) is violated for $V_{\mbox{Apel}}\cap H_0^1(\Omega^c)$ (see T. Apel in \cite{AN:2001,ANJ:2001}) define by 
 \begin{eqnarray*}
  V_{\mbox{Apel}}:=\left\{v_h\in L^2(\Omega):  \mbox{     } 
  {v_h}_{|K}\in \mbox{span}\{1,x,y,y^2\} \mbox{  } \forall  K\in \cT_h, \mbox{  } \int_{E}[v_h]_E=0  \forall  E\in\cE_h\right\};
 \end{eqnarray*}
therefore we had to enlarge the discrete space $V_h$ 
(i.e. $V_{\mbox{Apel}}\cap H_0^1(\Omega^c)
\subset V_h$).

\subsection{Crouzeix-Raviart elements III} Here we make the same restriction as in the previous section, i.e. we consider a triangulation 
of $\Omega$ made of rectangles. For the previous element, the local space ${V_h}_{|K}$ depends on the stretching direction of the rectangle
$K$. Here we modify the element such that this dependence on the directionality is removed.

Consider the reference rectangle $\bar{K}=(0,1)^2$, set $\bar{\mathcal{P}}:=\mathbb{P}^2$, and define the degrees of freedom (with the same 
notation as before) by
\begin{eqnarray*}
 \bar{\theta}_i(q):=\int_{\bar{E}_i}q, \mbox{  } i=1,\ldots,4,\mbox{    } \bar{\theta}_5(q):=\int_{\bar{K}} \bar{q}_5q, 
 \mbox{    } \bar{\theta}_6(q):=\int_{\bar{K}}q,
\end{eqnarray*}
with $\bar{q}_5$ as above. One easily checks that the triplet $\left(\bar{K},\bar{\mathcal{P}},\{\bar{\theta}_i\}_{i=1}^6\right)$ is a finite 
element (cf. \cite[Page 75]{AE05}) whose associated basis is given by 
$\{\bar{q}_i\}_{i=1}^6$, with 

$$ \bar{q}_1(\bar{x},\bar{y}):=1-4\bar{y}+3\bar{y}^2, \mbox{   } \hspace*{1cm}
\bar{q}_2(\bar{x},\bar{y}):=-2\bar{y}+3\bar{y}^2, $$
$$\hspace*{1cm}\bar{q}_3(\bar{x},\bar{y}):=1-4\bar{x}+3\bar{x}^2, \mbox{   } \hspace*{1cm}
\bar{q}_4(\bar{x},\bar{y}):=-2\bar{x}+3\bar{x}^2.
$$
$$\hspace*{1cm}\bar{q}_5(\bar{x},\bar{y}):=(2\bar{x}-1)(2\bar{y}-1), \mbox{   } \hspace*{1cm}
\bar{q}_6(\bar{x},\bar{y}):=6(\bar{x}-\bar{x}^2+\bar{y}-\bar{y}^2)-1.
$$
On a stretched rectangle $K$ we take the finite element $\left(K,\mathbb{P}^2,\{\theta_i\}_{i=1}^6\right)$ obtained by a standard affine 
transformation from $\bar{K}$ to $K$, i.e. 
$q_i(x,y)=\bar{q}_i(\bar{x},\bar{y})$ and $\theta_i(q)=\bar{\theta}_i(\bar{q})$.
The Assumption \ref{ass2} (i.e. (H.1) and (H.2) conditions) with the Crouzeix-Raviart condition are satisfied trivially\cite{creuse:03},
where the discrete space $V_h$ is defined by 
\begin{eqnarray*}
 V_h&:=& \left\{v_h\in L^2(\Omega): \mbox{   } {v_h}_{|K}\in\mathbb{P}^2(K) , \forall K\in \cT_h, \int_E [v_h]_E=0 \forall  E\in\cE_h\right\}\cap
 H_0^1(\Omega^c),
\end{eqnarray*}
and the bilinear form (resp. the discrete norm)
$J(.,.)$ (resp. $\lVert\cdot\lVert_h$) are defined as above.
\subsection{$\mathbb{Q}^k$ elements ($H_0^1(\Omega)$-Conforme approximation)}
We finally present an element currently used in $hp$ finite element approximations of corner and/or edge singularities as well as boundary 
layers, and achieving robust exponential convergence. We consider either a $2D$ triangulation of $\Omega$ made of 
triangles or rectangles. The discrete space is defined  for $k\geqslant 2$ by 
\begin{eqnarray}
 V_h:=\left\{v_h\in H_0^1(\Omega):  v_h=0 \mbox{  on  } \partial \Omega_c, {v_h}_{|K}\in \mathcal{P}_K^k, \forall K\in \mathcal{T}_h\right\}
 \subset V,
\end{eqnarray}
where $\mathcal{P}_K^k=\mathbb{Q}^k(K)$ if $K$ is a rectangle and $\mathcal{P}_K^k=\mathbb{P}^k(K)$ if $K$ is a triangle.
The Assumption \ref{ass2} and Crouzeix-Raviart property (CR) are clearly satisfied by definition \cite{creuse:03}.
\section{Error estimators}\label{Estimatorssection}
In order to  solve the coupled problem (\ref{model})-(\ref{boundary}) by efficient adaptive finite element methods, reliable and efficient 
a posteriori error analysis is important to provide appropriated indicators. 
In this section, we first define the local and global indicators and then the lower and upper error  bounds 
are derived.
\subsection{Residual error estimator}
The general philosophy of residual error estimators is to estimate an appropriate norm of the correct residual by terms that 
can be  evaluated easier, and that involve the data at hand. To this end define the exact element residuals:
\begin{dfn}\normalfont(\textbf{Exact element residuals}) Let $v_h\in V_h$ be an arbitrary finite element function. The exat 
element residuals over a triangle or 
rectangle $K\in\cT_h$ and over face $E\subset \overline{\Omega^c}$ are defined by
\begin{eqnarray}
 R_K(v_h):=f^m+\dive (\mathbb{K}\nabla v_h^m)-\alpha_{ex}(v_h^m-v_h^c)\delta_y, 
\end{eqnarray}
\begin{eqnarray}
 R_{E}(v_h):=f^c+\frac{d}{dx}\left(D\frac{dv_h^c}{dx}\right)+\alpha_{ex}({v_h^m}_{|y=0}-v_h^c),
\end{eqnarray}
respectively.

\end{dfn}
As it is common, these exact residuals are replaced by some finite-dimensional approximation called approximate element 
residual $r_{K}(v_h)$ and $r_{E}(v_h)$:
$$ r_{K}(v_h)\in\mathcal{P}_K^k \mbox{  on  } K\in \mathcal{T}_h \mbox{  and   } 
 r_{E}(v_h) \in\mathcal{P}_E^r\mbox{  on   }  E\subset \overline{\Omega^c};  (k,r)\in\mathbb{N}^2.
$$
This approximation is here  achieved by  projecting 
 $f^c$ on the space of piecewise constant functions in $\Om^c$
and picewise $\mathcal{P}_{K}^1$ functions  in $\Om^m$ for $f^m$, more precisely 
for each $E\subset\overline{\Omega^c}$, we take
\[
f^c_{E}=\frac{1}{|E|}\int_E f^c(\tau)d\tau,
\]
and for all $K\in\cT_h$ we take $f^m_K$ as the unique element of $\mathcal{P}_{K}^1$
such that
\[
\int_Kf^m_K(x,y)q(x,y)dxdy=\int_K f^m(x,y) q(x,y)dxdy, \forall q\in \mathcal{P}_{K}^1.
\]
We recall that $\mathcal{P}_{K}^k=\mathbb{P}^k(K)$ if $K$ is triangle, with for $E\in\cE(K)$  $\mathcal{P}_{E}^r=\mathbb{P}^r(E)$. Also 
$\mathcal{P}_{K}^k=\mathbb{Q}^k(K)$ if $K$ is rectangle, where for  $E\in\cE(K)$  $\mathcal{P}_{E}^r=\mathbb{Q}^r(E)$; $(k,r)\in\mathbb{N}^2$.

Thereby, we define the approximate element residuals.
\begin{dfn}\normalfont(\textbf{Approximate element residuals})
  Let $v_h\in V_h$ be an arbitrary finite element function. The approximate element residuals  are defined by
  \begin{eqnarray}
   r_{K}(v_h):=f_K^m+\dive (\mathbb{K}\nabla v_h^m)-\alpha_{ex}(v_h^m-v_h^c)\delta_y, \forall K\in\cT_h,
  \end{eqnarray}
  and 

\begin{eqnarray}
 r_{E}(v_h):=f^c_{E}+\frac{d}{dx}\left(D\frac{dv_h^c}{dx}\right)+
 \alpha_{ex}({v_h^m}_{|y=0}-v_h^c), \forall E\subset\overline{\Omega^c}.
\end{eqnarray}
\end{dfn}
We can now define the residual error estimators.
\begin{dfn}
 \normalfont (\textbf{Residual error estimators})
 For a conforming discretization, the  local residual error estimators are defined by
 \begin{eqnarray}\nonumber\label{conformingestimator}
  \Theta_K^2(u_h)&:=&h_{\min,K}^2\lVert r_K(u_h)\lVert_K^2+\displaystyle\sum_{E\in\cE_h
  (\partial K\cap \Omega^m)}\frac{h_{\min,K}^2}{h_E}
  \lVert[\mathbb{K}\nabla u_h^m\cdot \textbf{n}_E]_E\lVert_E^2\\
  &+&
  \displaystyle\sum_{E\in\cE_h(\partial K\cap \overline{\Omega^c})}\frac{h_{\min,K}^2}{h_E}
  \lVert r_{E}(u_h)\lVert_E^2.
 \end{eqnarray}
 For a non-conforming discretization, we set,
 \begin{eqnarray}\nonumber\label{thetanc}
  \Theta_K^2(u_h)&:=&h_{\min,K}^2\lVert r_K(u_h)\lVert_K^2+\displaystyle\sum_{E\in\cE_h
  (\partial K\cap \Omega^m)}\frac{h_{\min,K}^2}{h_E}
  \lVert[\mathbb{K}\nabla u_h^m\cdot \textbf{n}_E]_E\lVert_E^2\\
  &+&
  \displaystyle\sum_{E\in\cE_h(\partial K\cap \overline{\Omega^c})}\frac{h_{\min,K}^2}{h_E}
  \lVert r_{E}(u_h)\lVert_E^2\\\nonumber
  &+&\displaystyle\sum_{E\in\cE_h(\partial K\cap \Omega^m)}\frac{h_E}{h_{\min,K}^2}
  \lVert [u_h^m]_E\lVert_E^2.
 \end{eqnarray}

\end{dfn}
The global residual error estimator is given by 
\begin{eqnarray}
 \Theta (u_h)&:=& \left(\displaystyle\sum_{K\in \cT_h}\Theta_K(u_h)^2\right)^{1/2}.
\end{eqnarray}
Furthermore denote the local and global approximation terms by 
\begin{eqnarray*}\label{termeapproximationlocal}
 \zeta_{K}^2:=h_{\min,K}^2 \lVert R_K(u_h)-r_K(u_h)\lVert_K^2+
 \displaystyle\sum_{E\in\cE_h(\partial K\cap \overline{\Omega^c})}\frac{h_{\min,K}^2}{h_E}
 \rVert R_{E}(u_h)-r_{E}(u_h)\rVert_E^2.
\end{eqnarray*}
and
\begin{eqnarray}\label{termeapproximationglobal}
 \zeta &:=&\left(\displaystyle\sum_{K\in\mathcal{T}_h}\zeta_K^2\right)^{1/2}.
\end{eqnarray}
\begin{rmq}
 The residual character of each term on the right-hand sides of (\ref{conformingestimator}) and (\ref{thetanc}) is quite clear since if 
 $u_h$ would be the exact solution of (\ref{FF}), then they would vanish.
\end{rmq}


\subsection{Proof of the lower error bound}
 To prove local efficiency for $\omega\subset \Omega$ and $v\in V\cup V_h$, let us denote by 
 \begin{eqnarray}
  \lVert v\lVert_{h,\omega}^2:=\displaystyle\sum_{K\subset \bar{\omega}\cap\bar{\Omega}^m}\lVert v^m\lVert_{1,K}^2+\lVert v^c\lVert_{1,\bar{\omega}\cap 
  \bar{\Omega}^c}^2
  +\displaystyle\sum_{K\subset \bar{\omega}}J_K(v^m,v^m),
 \end{eqnarray}
where
$$J_K(v^m,v^m):=\displaystyle\sum_{E\in\cE_h(\partial \Omega^m\cap\partial K)}\frac{h_E}{h_{\min,E}^2}\cdot\lVert [v^m]_E\lVert_E^2,$$
for non-conforming discretization, and we set 
\begin{eqnarray}
  \lVert v\lVert_{h,\omega}^2:=\lVert v^m\lVert_{1,\omega\cap\bar{\Omega}^m}^2+\lVert v^c\lVert_{1,\omega\cap 
  \bar{\Omega}^c}^2,
\end{eqnarray}
for conforming discretization.   

The error estimator $\Theta(u_h)$ is consider efficient if it satisfies the following theorem:
\begin{thm}(\textbf{Local lower error bound})\label{bi}
 Let $u\in V$ be the exact solution and $u_h\in V_h$ be the finite element solution. Assume that the 
 Assumption \ref{ass1} holds. Then, the error is bounded locally from below for all $K\in\cT_h$ by
 \begin{eqnarray}\label{Local}
  \Theta_K(u_h)&\lesssim& \rVert u-u_h\rVert_{h,\tilde \omega_K}+
  \displaystyle\sum_{K'\subset \tilde \omega_K}\zeta_{K'},
 \end{eqnarray}
where $\tilde \omega_K$ is a finite union of neighbording elements of $K$.
\end{thm}
\begin{proof}
 We begin by bounding each the residuals separately. \\\\
 $\bullet$ \textbf{ Element residual in $\Omega^m$:}
 We start with the norm $\lVert r_K(u_h)\lVert_K$ of the element residual $r_K=
 r_K(u_h):=f_K^m+\dive (\mathbb{K}\nabla u_h^m)-\alpha_{ex}(u_h^m-u_h^c)\delta_y$. Since we use linear or bilinear polynomial functions, 
 $r_K\in\mathcal{P}_K^k$ holds for certain $k\in\mathbb{N}$. For $\textbf{x}\in K$ let 
 \begin{eqnarray}
  w_K(\textbf{x}):=r_K(u_h)(\textbf{x})\cdot b_K(\textbf{x})\in H_0^1(K),
 \end{eqnarray}
where the element bubble function $b_K$ is from Section \ref{bulle}. Integration by parts yields
\begin{eqnarray*}
 \int_Kr_K\cdot w_K&=&\int_K[\dive(\mathbb{K}\nabla u_h^m)+\alpha_{ex}(u_h^m-u_h^c)\delta_y+f^m]\cdot w_K+\int_K(f_K^m-f^m)\cdot w_K\\
 &=&-\int_K(\mathbb{K}\nabla u_h^m)\cdot\nabla w_K+\int_K f^m\cdot w_K+\int_{K\cap\Omega^c}[\alpha_{ex}(u_h^m-u_h^c)\delta_y]\cdot w_K\\
 &+&\int_K(f_K^m-f^m)\cdot w_K
\end{eqnarray*}
We use the weak formulation (\ref{FF}) to obtain,
\begin{eqnarray*}
 \int_Kr_K\cdot w_K&=&
 \int_K[\mathbb{K}\nabla (u^m-u_h^m)]\cdot\nabla w_K\\
 &-&\int_{K\cap\Omega^c}\alpha_{ex}[(u^m-u_h^m)-(u^c-u_h^c)]\delta_y\cdot w_K\\
 &+&\int_K(f_K^m-f^m)\cdot w_K
\end{eqnarray*}
Hence
\begin{eqnarray*}
 \int_Kr_K\cdot w_K &\leqslant&
 \lVert \mathbb{K}\nabla (u^m-u_h^m)\lVert_K\cdot\lVert \nabla w_K\lVert_K\\
 &+&\alpha_{ex}\left(\lVert u^m-u_h^m\lVert_{K\cap\Omega^c}+\lVert u^c-u_h^c 
 \lVert_{K\cap\Omega^c}\right)\cdot\lVert w_K\lVert_{K\cap\Omega^c}\\
 &+&
 \lVert R_K(u_h)-r_K(u_h)\lVert_K\cdot \lVert w_K\lVert_K
\end{eqnarray*}
Recalling (\ref{I1}), (\ref{I2}), and $0\leqslant b_K\leqslant 1$ gives the following bounds, 
\begin{eqnarray*}
 \left|\int_Kr_K\cdot w_K\right|&=&\lVert b_K^{1/2}\cdot r_K\lVert_K^2\sim \lVert r_K\lVert_K^2\\
 \lVert\nabla w_K\lVert_K&=&\lVert\nabla (b_K\cdot r_K)\lVert_K\lesssim h_{\min,K}^{-1}\cdot\lVert r_K\lVert_K\\
 \lVert w_K\lVert_K&=&\lVert b_K\cdot r_K\lVert_K\leqslant \lVert r_K\lVert_K,
\end{eqnarray*}
that result in 
\begin{eqnarray}\label{bi1}
 h_{\min,K}^2\cdot\lVert r_K\lVert_K^2&\lesssim&\lVert u-u_h\lVert_{h,K}^2+\zeta_K^2.
\end{eqnarray}
$\bullet$ \textbf{ Normal jump in $\Omega^m$:} Now we aim at a bound of the term\\ $\displaystyle\sum_{E\in\cE_h
  (\partial K\cap \Omega^m)}\frac{h_{\min,K}^2}{h_E}
  \lVert[\mathbb{K}\nabla u_h^m\cdot \textbf{n}_E]_E\lVert_E^2$ of the gradient jump across some inner face 
  $E\subset\Omega^m$. We fix $E\in \cE_h
  (\Omega^m)$. Since we use linear or bilinear polynomial functions,\\ $[\mathbb{K}\nabla u_h^m\cdot \textbf{n}_E]_E\in\mathcal{P}_E^r$ 
  holds for certain $r\in\mathbb{N}$. Let 
  $K_1$ and $K_2$ be the two elements that $E$ belongs to. The right hand side 
  $f^m=\alpha_{ex}(u^m-u^c)\delta_y-\dive (\mathbb{K}\nabla u^m)$ is assumed to be in $L^2(\Omega^m)$. Integration by parts yields for any 
  function $w_E\in H_0^1(W_E)$
  \begin{eqnarray*}
   0&=&-\int_{w_E}(\mathbb{K}\nabla u^m)\cdot \nabla w_E-\int_{w_E\cap\Omega^c}\alpha_{ex}(u^m-u^c)\delta_y\cdot w_E\\
   &+&\int_{w_E}f^m\cdot w_E\\
   -\int_Ew_E\cdot [\mathbb{K}\nabla u_h^m\cdot \textbf{n}_E]_E&=&
   \displaystyle\sum_{i=1}^2\int_{\partial K_i}w_E\cdot (\mathbb{K}\nabla u_h^m\cdot \textbf{n}_E)\\
   &=&\displaystyle\sum_{i=1}^2\left(\int_{K_i}\mathbb{K}\nabla u_h^m\cdot\nabla w_E+
   \int_{K_i}w_E\cdot\dive (\mathbb{K}\nabla u_h^m)\right)\\
   &=&
   \displaystyle\sum_{i=1}^2\left(\int_{K_i}\mathbb{K}\nabla u_h^m\cdot\nabla w_E-\int_{K_i}f_{K_i}^m\cdot w_E\right.\\
  &+&\left. \alpha_{ex}\int_{K_i\cap\Omega^c}
   (u_h^m-u_h^c)\delta_y\cdot w_E\right)\\
   &=&
   \displaystyle\sum_{i=1}^2\left(-\int_{K_i}\mathbb{K}\nabla (u^m-u_h^m)\cdot\nabla w_E\right.\\
  &-&\left. \alpha_{ex}\int_{K_i\cap\Omega^c}
   [(u^m-u_h^m)-(u^c-u_h^c)]\delta_y\cdot w_E\right.\\
   &+&\left.\int_{K_i}(f^m-f_{K_i}^m)\cdot w_E
   \right).
   \end{eqnarray*}
Let now the function $w_E\in H_0^1(w_E)$ be defined by 
\begin{eqnarray}
 w_E:=F_{ext}(-[\mathbb{K}\nabla u_h^m\cdot \textbf{n}_E]_E)\cdot b_E,
\end{eqnarray}
with $F_{ext}$ being the extension operator of Section \ref{bulle}, and $b_E$ being the face bubble function.
Because  of ${w_E}_{|E}={[\mathbb{K}\nabla u_h^m\cdot \textbf{n}_E]_E\cdot b_E}_{|E}$, we conclude 
\begin{eqnarray*}
 \left\lVert[\mathbb{K}\nabla u_h^m\cdot \textbf{n}_E]_E\cdot b_E^{1/2} \right\lVert_E^2\lesssim
 \displaystyle\sum_{i=1}^2\left\{
 \left\lVert \mathbb{K}\nabla (u^m-u_h^m)\lVert_{K_i}\cdot\lVert\nabla w_E\right\lVert_{K_i}\right.\\
 \left.+\left(\alpha_{ex}\lVert u^m-u_h^m\lVert_{K_i\cap\Omega^c}+\lVert u^c-u_h^c\lVert_{K_i\cap\Omega^c}\right)\cdot
 \lVert w_E\lVert_{K_i\cap\Omega^c}
 +\lVert f^m-f_{K_i}^m \lVert_{K_i}\cdot\lVert w_E\lVert_{K_i}\right\}.
\end{eqnarray*}
The function $w_E$ is piecewise cubic on $K_1\cup K_2$. The equivalence relations (\ref{I3})-(\ref{I4}) imply 
\begin{eqnarray*}
\int_Ew_E\cdot [\mathbb{K}\nabla u_h^m\cdot \textbf{n}_E]_E&=&
\int_E[\mathbb{K}\nabla u_h^m\cdot \textbf{n}_E]_E^2\cdot b_E\\
&=&\lVert[\mathbb{K}\nabla u_h^m\cdot \textbf{n}_E]_E\cdot b_E^{1/2} \lVert_E^2\\
&\sim& \lVert[\mathbb{K}\nabla u_h^m\cdot \textbf{n}_E]_E\lVert_E^2\\
 \left\lVert \nabla (F_{\mbox{ext}}([\mathbb{K}\nabla u_h^m\cdot \textbf{n}_E]_E)\cdot b_E)\right\lVert_{K_i}      &=&\lVert\nabla (w_E)\lVert_{K_i}\\
&\sim& h_E^{1/2}h_{\min,K_i}^{-1}\cdot \lVert[\mathbb{K}\nabla u_h^m\cdot \textbf{n}_E]_E\lVert_E\\
 \left\lVert F_{\mbox{ext}}([\mathbb{K}\nabla u_h^m\cdot \textbf{n}_E]_E)\cdot b_E\right\lVert_{K_i}&=&\lVert w_E\lVert_{K_i}
\sim h_E^{1/2}\cdot \lVert[\mathbb{K}\nabla u_h^m\cdot \textbf{n}_E]_E\lVert_E
\end{eqnarray*}
and subsequently lead to
\begin{eqnarray*}
 \lVert[\mathbb{K}\nabla u_h^m\cdot \textbf{n}_E]_E\lVert_E^2\lesssim 
 \displaystyle\sum_{i=1}^2\left\{
 \lVert \mathbb{K}\nabla (u^m-u_h^m)\lVert_{K_i}\cdot h_E^{1/2}h_{\min,K_i}^{-1}
 \lVert[\mathbb{K}\nabla u_h^m\cdot \textbf{n}_E]_E\lVert_E\right.\\
 \left.+\left(\alpha_{ex}\lVert u^m-u_h^m\lVert_{K_i\cap\Omega^c}+\lVert u^c-u_h^c\lVert_{K_i\cap\Omega^c}\right)\cdot
 h_E^{1/2}\lVert[\mathbb{K}\nabla u_h^m\cdot \textbf{n}_E]_E\lVert_E\right.\\
 \left.+\lVert f^m-f_{K_i}^m\lVert_{K_i}\cdot h_E^{1/2}\lVert[\mathbb{K}\nabla u_h^m\cdot \textbf{n}_E]_E\lVert_E
 \right\}.
\end{eqnarray*}
The dimensions $h_E\sim h_{E,K_i}$ and $h_{\min,K_i}$ cannot change rapidly for adjacent element. Thereby,
\begin{eqnarray*}
 \lVert[\mathbb{K}\nabla u_h^m\cdot \textbf{n}_E]_E\lVert_E\lesssim 
 \frac{h_E^{1/2}}{h_{\min,E}} \cdot 
 \left\{\lVert\nabla (u^m-u_h^m) \lVert_{W_E}\right.\\
 \left.+ \left(\alpha_{ex}\lVert u^m-u_h^m\lVert_{K_i\cap\Omega^c}+\lVert u^c-u_h^c\lVert_{K_i\cap\Omega^c}\right)\cdot
 h_{\min,E}+\lVert f^m-f_{K_i}^m\lVert_{K_i}\cdot h_{\min,E}
 \right\}.
\end{eqnarray*}

For a fixed element $K=K_1$ we sum up over all (inner) faces $E\in \cE_h(\Omega^m\cap K)$ and obtain,
\begin{eqnarray}\label{bi2}
\displaystyle\sum_{E\in\cE_h
  (\partial K\cap \Omega^m)}\frac{h_{\min,K}^2}{h_E}
  \lVert[\mathbb{K}\nabla u_h^m\cdot \textbf{n}_E]_E\lVert_E^2\lesssim 
\lVert u-u_h \lVert_{h,W_K}^2+\displaystyle\sum_{K'\subset W_K}\zeta_{K'}^2.
\end{eqnarray}
$\bullet$ \textbf{ Element residual in $\Omega^c$:} Let $K\in\cT_h$. Next the term $\frac{h_{\min,K}^2}{h_E}
  \lVert r_{E}(u_h)\lVert_E^2$ for a face $E\in\cE_h(\partial K\cap \overline{\Omega^c})$ of the pipe-flow region boundary 
  is to be bounded. Let $E\subset \overline{\Omega^c}$ and we design by 
  $K_1$ and $K_2$ be the two elements that $E$ belongs to (i.e. $W_E=K_1\cup K_2$). 
   Since we use linear or bilinear ansatz functions, \\$r_E=r_{E}(u_h):=f^c_{E}+\frac{d}{dx}\left(D\frac{du_h^c}{dx}\right)+
 \alpha_{ex}({u_h^m}_{|y=0}-u_h^c)$, 
 $r_E\in\mathcal{P}_E^r$ holds for some $r\in\mathbb{N}$.
  We set:
  \begin{eqnarray}
   w_E:=(F_{\mbox{ext}} (r_E(u_h))\cdot b_E),
  \end{eqnarray}
  and we assume that $w_E=0$ on $\overline{\Omega^m}\backslash \Omega^c$.
We use the weak formulation (\ref{FF}) to obtain,
\begin{eqnarray*}
 0=\int_{w_E\cap \Omega^c} D\frac{du^c}{dx}\cdot\frac{dw_E}{dx}-\alpha_{ex}\int_{\Omega^c\cap W_E}(u^m-u^c)\cdot w_E-\int_{W_E} f^c\cdot w_E.
\end{eqnarray*}
Integration by parts yields to 
\begin{eqnarray*}
 \int_E r_E\cdot w_E&=&\int_{W_E} f_E^c\cdot w_E+\int_{W_E}\frac{d}{dx}\left(D\frac{du_h^c}{dx}\right)\cdot w_E+\alpha_{ex}\int_{W_E\cap \Omega^c}
 (u_h^m-u_h^c)\cdot w_E\\
 &=&\int_{W_E} f_E^c\cdot w_E-\int_{W_E}D\frac{du_h^c}{dx}\cdot \frac{dw_E}{dx}+\alpha_{ex}\int_{W_E\cap \Omega^c}
 (u_h^m-u_h^c)\cdot w_E\\
 &=&\int_{W_E}-(f^c-f_E^c)\cdot w_E-\int_{W_E}D\frac{d(u^c-u_h^c)}{dx}\cdot \frac{dw_E}{dx}\\
 &-&\alpha_{ex}\int_{W_E\cap \Omega^c}
 [(u^m-u_h^m)-(u^c-u_h^c)]\cdot w_E\\
\end{eqnarray*}
Because ${w_E}_{|E}={r_E\cdot b_E}_{|E}$ we conclude by Cauchy-Schwarz inequality,
\begin{eqnarray*}
 \int_Er_E^2\cdot b_E&\leqslant&\lVert f^c-f_E^c\lVert_{W_E}\cdot \lVert w_E\lVert_{W_E}+D\left\lVert \frac{d(u^c-u_h^c)}{dx}\right\lVert_{W_E}
 \cdot \left\lVert\frac{dw_E}{dx}\right\lVert_{W_E}\\
 &+&\alpha_{ex}\left(\lVert u^m-u_h^m\lVert_{W_E\cap \Omega^c}+\lVert u^c-u_h^c\lVert_{W_E\cap \Omega^c}\right)\cdot
 \lVert w_E\lVert_{W_E\cap \Omega^c}.
\end{eqnarray*}
We fix $K=K_1$. Thus by inverse inequalities (\ref{I3})-(\ref{I5}), we deduce the estimation:
\begin{eqnarray}\label{bi3}
 \displaystyle\sum_{E\in\cE_h(\Omega^c\cap K)}\frac{h_{\min,K}^2}{h_E}
  \lVert r_{E}(u_h)\lVert_E^2\lesssim\lVert u-u_h \lVert_{h,W_K}^2+\displaystyle\sum_{K'\subset W_K}\zeta_{K'}^2.
\end{eqnarray}
$\bullet$ \textbf{ Nonconforming element:} It remains now to estimate the local indicator 
$\displaystyle\sum_{E\in\cE_h(\partial K\cap \Omega^m)}\frac{h_E}{h_{\min,K}^2}
  \lVert [u_h^m]_E\lVert_E^2$. Because the jump of $u\in H_0^1(\Omega)$ is zero through all the edges of $\Omega$, we clearly have 
  \begin{eqnarray}\nonumber\label{bi4}
   \displaystyle\sum_{E\in\cE_h(\partial K\cap \Omega^m)}\frac{h_E}{h_{\min,K}^2}
  \lVert [u_h^m]_E\lVert_E^2\leqslant J_K(u_h^m,u_h^m)&=&J_K(u^m-u_h^m,u^m-u_h^m)\\
  &\leqslant&\lVert u^m-u_h^m\lVert_{h,K}^2.
  \end{eqnarray}
Summarising, the estimates (\ref{bi1}), (\ref{bi2}), (\ref{bi3}) and (\ref{bi4}) provide the desired local lower error bound of Theorem \ref{bi}.
\end{proof}

\subsection{Proof of the upper error bound}
The main result of this subsection can be stated as follows. 
\begin{thm}\label{bs}
 (\textbf{Upper error bound-conforming case}) Assume a conform discretization (i.e. $V_h\subset V$).
 Let $u\in V$ be the exact solution and $u_h\in V_h$ be the finite element solution. Assume that the 
 Assumptions \ref{ass1} et \ref{ass2} hold. Then the  error is bounded globally from above by
 \begin{eqnarray}\label{bornesup}
  \rVert u-u_h\rVert_h&\lesssim& m_1(u-u_h,\cT_h)\cdot\left[\Theta(u_h)^2+\zeta^2\right]^{1/2}.
 \end{eqnarray}
\end{thm}
\begin{proof}
  In order to derive (\ref{bornesup}) we utilize the orthogonality property of the error 
  $$a(u-u_h,v_h)=0 \mbox{    } \forall v_h\in V_h.$$
  Let $v_h\in V_h$. Integration by parts, triangle inequality and the weak formulation (\ref{FF}) give for all $v\in V$, 
  \begin{eqnarray*}
   a(u-u_h,v)&=&a(u-u_h,v-v_h)\\
   &=& (f^m,v^m-v_h^m)_m+(f^c,v^c-v_h^c)_c-a(u_h,v-v_h)\\
   &=&(f^m,v^m-v_h^m)_m+(f^c,v^c-v_h^c)_c-\int_{\Omega^m}\mathbb{K}\nabla u_h^m\cdot\nabla (v^m-v_h^m)dxdy\\
   &-&\int_{0}^L D\frac{du_h^c}{dx}\cdot
   \frac{d(v^c-v_h^c)}{dx} dx\\
   &-&\alpha_{ex}\int_{0}^L(u_h^m(x,0)-u_h^c(x))(v^m-v_h^m)(x,0))dx\\
   &+&\alpha_{ex}\int_{0}^L(u_h^m(x,0)-u_h^c(x))(v^c-v_h^c)(x))dx\\
   &\leqslant&\displaystyle\sum_{K\in\cT_h}\left\{\int_K|(f^m-f_K^m)(v^m-v_h^m)|\right.\\
   &+&\left.\int_K |r_K(u_h)(v^m-v_h^m)|+\displaystyle\sum_{E\in
   \cE_h(K)}\int_E|[\mathbb{K}\nabla u_h^m\cdot \textbf{n}_E]_E(v^m-v_h^m)|\right.\\
   &+&\left.\displaystyle\sum_{E\in
   \cE_h(K\cap\Omega^c)}\int_E|r_{E}(u_h)(v^c-v_h^c)|\right.\\
   &+&\left.\displaystyle\sum_{E\in
   \cE_h(K\cap\overline{\Omega^c})}\int_E |(f^c-f_{E}^c)(v^c-v_h^c)| \right\}\\
   &\leqslant&\left(\displaystyle\sum_{K\in\cT_h}h_{\min,K}^2\lVert r_K(u_h) \lVert_K^2\right)^{1/2}
   \cdot\left(\displaystyle\sum_{K\in\cT_h}h_{\min,K}^{-2}\lVert v^m-v_h^m \lVert_K^2\right)^{1/2}\\
   &+& \left(\displaystyle\sum_{E\in
   \cE_h(\Omega^m)}\frac{h_{\min,E}^2}{h_E}\lVert[\mathbb{K}\nabla u_h^m\cdot \textbf{n}_E]_E
   \lVert_E^2\right)^{1/2}\cdot
   \left(\displaystyle\sum_{E\in
   \cE_h(\Omega^m)}\frac{h_E}{h_{\min,E}^2}\lVert v^m-v_h^m\lVert_E\right)^{1/2}\\
   &+& \left(\displaystyle\sum_{E\in
   \cE_h(\Omega^c)}\frac{h_{\min,E}^2}{h_E}\lVert r_{E}(u_h)
   \lVert_E^2\right)^{1/2}\cdot
   \left(\displaystyle\sum_{E\in
   \cE_h(\Omega^c)}\frac{h_E}{h_{\min,E}^2}\lVert v^c-v_h^c\lVert_E\right)^{1/2}\\
    &+& \left(\displaystyle\sum_{K\in\cT_h} h_{\min,K}^2\cdot\lVert R_K(u_h)-r_K(u_h)\lVert_T^2\right)^{1/2}
    \cdot\left(\displaystyle\sum_{K\in\cT_h}h_{\min,K}^{-2}\lVert v^m-v_h^m \lVert_K^2\right)^{1/2}\\
    &+&\left(\displaystyle\sum_{E\in
   \cE_h(\Omega^c)}\frac{h_{\min,E}^2}{h_E}\lVert R_{E}(u_h)-r_E(u_h)
   \lVert_E^2\right)^{1/2}\cdot
   \left(\displaystyle\sum_{E\in
   \cE_h(\Omega^c)}\frac{h_E}{h_{\min,E}^2}\lVert v^c-v_h^c\lVert_E\right)^{1/2}
   \end{eqnarray*}
Every second root term is bounded by $m_1(v,\cT_h)\cdot\lVert v\lVert_h$ by means of the interpolation Lemma \ref{Interpolationanisotrope}. 
Substituting $v:=u-u_h$, then the  coercivity of $a$ (i.e. $a(u-u_h,u-u_h)\gtrsim \lVert u-u_h\lVert_h^2$) yields an upper 
bound of the error (\ref{bornesup}). 
\end{proof}
The upper error bound for non-conforming case on anisotropic meshes will derive as \cite{AHN:15} in a forthcoming paper. 
The Section \ref{conclusion} gives the 
procedure of proving this non-conforming case. Nevertheless the upper error bound for non-conforming case on isotropic meshes is 
consummate in Section \ref{conclusion}.


\begin{com}\normalfont(\textbf{Upper error bound-conforming case})
 The upper error bound (\ref{bornesup}) contains an alignment measure $m_1(.,.)$. This is in contrast to estimators for isotropic
 meshes: For anisotropic discretizations, all known estimators are (explicitly or implicitly) based on an anisotropic mesh that is 
 suitably aligned with the anisotropic function. 
 Compared with the isotropic estimators, our upper error bound is special in the sense that the alignment measure cannot be 
 evaluated explicitly. However, this should not be considered too much as a disadvantage. For example, the alignment measure
  $m_1(e,.)$ for the error $e=u-u_h$ is of size $\mathcal{O}(1)$ for sufficiently good meshes \cite{ESK:04}.
  In pratical computations one may simply use the error estimator without considering the alignment measure\cite{ESK:04}. For adaptive algorithms 
  this is well justified since the lower error bound (\ref{Local}) holds unconditionally.
\end{com}
\subsection{Application to isotropic Discretization}\label{isotropic}
Since our analysis gives new results for on isotropic meshes, we here summarize them. On 
isotropic discretizations, our analysis holds with $h_{\min,K}\sim h_E\sim h_K$ for $E\in \cE(K)$ and the alignment measure 
$m_1(.,.)\sim 1$. In other words, the above results may be rephrased as follows: the residual error estimator is here given by
\begin{eqnarray}
 \Theta (u_h)&:=& \left(\displaystyle\sum_{K\in \cT_h}\Theta_K(u_h)^2\right)^{1/2},
\end{eqnarray}
with

\begin{eqnarray}\nonumber
\Theta_K^2(u_h)&:=&h_{K}^2\lVert r_K(u_h)\lVert_K^2+\displaystyle\sum_{E\in\cE_h
  (\partial K\cap \Omega^m)}h_E
  \lVert[\mathbb{K}\nabla u_h^m\cdot \textbf{n}_E]_E\lVert_E^2\\
  &+&
  \displaystyle\sum_{E\in\cE_h(\partial K\cap \overline{\Omega^c})}h_E
  \lVert r_{E}(u_h)\lVert_E^2.
 \end{eqnarray}
 for conforming discretization, and 
 \begin{eqnarray}\nonumber
  \Theta_K^2(u_h)&:=&h_{K}^2\lVert r_K(u_h)\lVert_K^2+\displaystyle\sum_{E\in\cE_h
  (\partial K\cap \Omega^m)}h_E
  \lVert[\mathbb{K}\nabla u_h^m\cdot \textbf{n}_E]_E\lVert_E^2\\
  &+&
  \displaystyle\sum_{E\in\cE_h(\partial K\cap \overline{\Omega^c})}h_E
  \lVert r_{E}(u_h)\lVert_E^2\\\nonumber
  &+&\displaystyle\sum_{E\in\cE_h(\partial K\cap \Omega^m)}h_E
  \lVert [u_h^m]_E\lVert_E^2,
 \end{eqnarray}
 for non-conforming discretization. The local and global approximation terms become: 
\begin{eqnarray*}\label{termeapproximationlocal}
 \zeta_{K}^2:=h_{K}^2 \lVert R_K(u_h)-r_K(u_h)\lVert_K^2+
 \displaystyle\sum_{E\in\cE_h(\partial K\cap \overline{\Omega^c})}h_E
 \rVert R_{E}(u_h)-r_{E}(u_h)\rVert_E^2.
\end{eqnarray*}
and
\begin{eqnarray}\label{termeapproximationglobal}
 \zeta &:=&\left(\displaystyle\sum_{K\in\mathcal{T}_h}\zeta_K^2\right)^{1/2}.
\end{eqnarray}
We recall that here, $h_K$ (resp. $h_E$) is the diameter of $K$ (resp. of $E$). With these definitions, the lower error bound (\ref{Local}) 
of Theorem \ref{bi} holds for isotropic elements $K$. On the other hand, the upper bound (\ref{bornesup}) of 
Theorem \ref{bs} reduces to
\begin{eqnarray}
    \parallel u-u_h\parallel_h&\lesssim&  
    \left[\Theta (u_h)^2+\zeta^2\right]^{1/2}.
 \end{eqnarray}

\section{Concluding remarks}\label{conclusion}
We have proposed and rigorously analysed a posteriori error estimate for the finite element approximation of a coupled continuum 
pipe-flow/Darcy model on anisotropic meshes. This model describes flow in porous media with an embedded conduit pipe. 
Our  investigations covers conforming and nonconforming discretizations, $2D$ domain as well as different kinds of standard elements. Much 
effort has been taken to impose as few assumptions as possible. For nonconforming discretizations, the main demand consists in 
Crouzeix-Raviart type elements. Different strategies are applied to estimate the lower and upper error bounds. These main results are 
summarized in Theorems \ref{bi} and \ref{bs}. In order to obtain sharp bound for reliability, the anisotropic mesh has to be properly aligned,
as it is the case with all known anisotropic (a posteriori) estimators. Here, this alignment enters explicitly via a so-called 
alignment measure. In addition, this mesh alignment is with respect to the error $e=u-u_h$. 
In contrast to upper error bound, the lower error bound (\ref{Local}) 
holds unconditionally. For isotropic discretizations, much of the analysis simplifies. The main results are presented in Section \ref{isotropic}
and the investigations seem to be novel.

However, many issues remain to be addresses in this area:\\
$\bullet$ \textbf{Upper error bound/nonconforming case.} We  give here the procedure of proving the nonconforming case for 
the upper error bound. To obtain the upper error bound for nonconforming case, Cl\'ement interpolation operator is not 
sufficient because additional term is included in the error estimator that measure the non-conformity of the method. In order to treat 
appropriately this non-conformity, we further need an estimate of the non-conforming error. Indeed, we can proved that 
(cf. \cite[Lemma 4.6]{AHN:15}),
\begin{eqnarray}\label{BNC}
 \lVert u-u_h\lVert_h&\lesssim& m_1(u-u_h,\cT_h)\cdot[\Theta(u_h)^2+\zeta^2]^{1/2}+\displaystyle\inf_{v_h\in V\cap V_h}\lVert u_h-v_h\lVert_h.
\end{eqnarray}
In this estimation, $\Theta(u_h)$ is the conformity estimator of the method given by (\ref{conformingestimator}) and the additionally term  
$\displaystyle\inf_{v_h\in V\cap V_h}\lVert u_h-v_h\lVert_h$ measures the non-conformity of the method. If the discretization is 
isotropic, the non-conformity term can be bounded by using Oswald interpolation (see \cite[Theorems 2.1 and 2.2]{karaka:03}):
\begin{eqnarray}
 \displaystyle\inf_{v_h\in V\cap V_h}\lVert u_h-v_h\lVert_h\lesssim \left[J(u_h,u_h)\right]^{1/2},
\end{eqnarray}
where $[J(u_h,u_h)]^{1/2}$ is the non-conformity estimator. Thus, we intend to proceed (in 
a forthcoming paper) as in \cite[Theorems 2.1 and 2.2]{karaka:03} (see also \cite[Theorem 3.3]{AHN:15}) while building an adapted 
anisotropic Oswald interpolation. We present also the results of numerical tests with the finite element methods.\\
$\bullet$ \textbf{CCPF model in 3D.} This work focuses on the 2D-CCPF model. In \cite[Section 4]{XW:00}, Xiaoming Wang  observes that the 
Mathematical problem with the original CCPF and Hua's CCPF is that the fluid exchanges occur on a very singular space: point singularity in the 
original CCPF case and line singularity in Hua's model. He proposes the following new CCPF model assuming the simple case of an one 
dimensional conduit centered at the x-axis and laminar flow:

\begin{eqnarray}
 \left\{
\begin{array}{ccccccccccccccc}\label{r}
 &S\frac{\partial u^m}{\partial t}-\nabla(\mathbb{K}\nabla u^m)& &=& &-\alpha_{ex}(u^m\delta_{\Gamma}-u^c\delta_{\Gamma})/|\Gamma_x|+
 R^m& &\mbox{  in  }&
   &\Omega^m&\\
   &-\frac{\partial}{\partial x}\left(D\frac{\partial u^c}{\partial x}\right)& &=& &\alpha_{ex} 
   \left(\frac{1}{|\Gamma_x|}\int_{\Gamma_x}u^m dl_x-u^c\right)+R^c& &\mbox{  in  }& & \Omega^c&,
\end{array}
\right.
\end{eqnarray}
 where $\Gamma$ is the boundary of the circular horizontal conduit centered at $x-$ axis, $\Gamma_x$ is the cross section of $\Gamma$ at $x$ 
 (a circle), $dl_x$ represents the infinitesimal increment of arc lengh on $\Gamma_x$ (equivqlent to 
 $r(x)d\theta$ in the cylindrical coordinates with $r(x)$ being the radius and $\theta$ being the angle), and $|\Gamma_x|$ is the length of 
 $\Gamma_x$ which is $\pi d(x)=2\pi r(x)$.
To treat this singularity, the anisotropic meshes are more adapted (see \cite{creuse:03}). That is why, we hope in a near further worked on this model 
proceeding similarly to \cite{Nadir:2007}.\\
$\bullet$ \textbf{Boussinesq equations.} Finally, we like to extend our results to Boussinesq equations with thermocapillarity effect on the 
surface and nonhomogeneous boundary conditions for the velocity and the temperature\cite{AP:15}.

\section{Acknowledgements}
The  author thanks African Institute for Mathematical Sciences (AIMS South Africa) for hosting him for a two months research visit.
We thank Serge Nicaise (UVHC, France) for his collaboration.  




\end{Large}
\end{document}